\documentclass[8pt]{article}
\usepackage[francais,english]{babel}
\usepackage[T1]{fontenc}
\usepackage{kpfonts}
\usepackage{color}\usepackage{color}
\usepackage{bera}
\usepackage{amssymb}
\usepackage{geometry}
\usepackage[utf8]{inputenc}
\usepackage{amsmath,mathtools}
\usepackage{graphicx}
\usepackage{sectsty}
\usepackage{lipsum}
\usepackage{mathrsfs}
\usepackage{amsthm}
\allsectionsfont{\centering}
\newtheorem{theorem}{Theorem}[section]

\newtheorem{example}{Example}[section]
\newtheorem{definition}{Definition}[section]
\newtheorem{proposition}{Proposition}[section]
\newtheorem{remark}{Remark}[section]

\numberwithin{equation}{section}
\begin{document}
\large
\begin{center}
   \textbf{A New Method For Solving Fractional  And Classical Differential Equations Based On a New Generalized  Fractional Power Series }
   \end{center}
   \hrule
   \begin{center}
   \textbf{Youness Assebbane}$^{(1,a)}$,
   \textbf{Mohamed Echchehira}$^{(1,b)}$, \textbf{Mohamed Bouaouid}$^{(2,c)}$ and \textbf{Mustapha Atraoui}$^{(1,d)}$
\end{center}
\begin{center}
$^{(1)}$ Ibn Zohr University, Ait Melloul University Campus, Faculty of Applied Sciences, Department of
Mathematics, Research Laboratory for Innovation in Mathematics and Intelligent Systems, BP 6146, 86150, Agadir, Morocco.\\
$^{(2)}$Sultan Moulay Slimane University, National School of Applied Sciences, 23000 Béni-Mellal, Morocco.
\end{center}
\begin{center}
$^{(a)}$ youness.assebbane.97@edu.uiz.ac.ma
\end{center}
\begin{center}
$^{(b)}$ mohamed.echchehira.57@edu.uiz.ac.ma
\end{center}
\begin{center}
$^{(c)}$ bouaouidfst@gmail.com
\end{center}
\begin{center}
$^{(d)}$ m.atraoui@uiz.ac.ma
\end{center}

\section*{Abstract}
 The main objective of this paper is to introduce an algorithm for solving fractional and classical differential equations based on a new generalized fractional power series. The algorithm relies on  expanding the solution of an FDE or an ODE as a generalized power series, shedding light on the choice of the exponent for the monomials. Furthermore, it accommodates situations where terms in the equation are multiplied by $t^{\alpha}$ for example. The key contribution is how the exponents for these terms are chosen, which is different from traditional methods.

\begin{description}
 
  \item[Keywords:] fractional power series; fractional-order differential equations; Riemann-Liouville fractional derivative; Caputo fractional derivative; Mittag-Leffler function.
  \end{description}
\section{Introduction}
In recent years, considerable interest has been shown in the so-called fractional calculus, which allows for integration and differentiation of any order, not necessarily integer. These methods have been used as excellent sources and tools to model many phenomena in various fields of engineering, science, and technology. Furthermore, these tools are also used in fields such as epidemiology \cite{Angstmann10}; fluid and solid mechanical models\cite{oliveira5}; compartment models \cite{Angstmann11}; economics and finance \cite{Tarasov15}; fractal media \cite{Henry14}; nerve cell signaling \cite{Matlob13}; viscoelastic materials \cite{Jaradat12}; and many more.\\

Laplace transform methods are applicable to linear fractional order differential equations with constant coefficients (see, for example,\cite{Lin}). However, this method faces limitations, especially when dealing with inhomogeneous equations, as it can be challenging to find closed-form Laplace transforms for certain functions or to invert these transforms into known computable functions. For linear fractional order differential equations with variable coefficients, an additional limitation arises: the Laplace transforms of products of functions can typically only be determined in exceptional cases.\\

In the context of linear differential equations with integer orders and variable coefficients, series expansion methods are well-recognized for their efficacy in finding solutions. Many studies have sought to extend these methods to linear fractional order differential equations, which are seen as a specific case of our algorithm, considered more effective in solving both fractional and classical differential equations involving variable coefficients.\\

The structure of this paper is as follows. In Section 2, we briefly recall some needed preliminaries. In Section 3, we describe a generalized fractional power series solution for fractional differential equations. In Section 4, we introduce a new algorithm for solving both fractional and classical differential equations and provide some examples illustrating the efficacy of this method. In Section 6, we summarize the findings and contributions of this paper.\\

\section{Definitions and Preliminaries}

 We define some essential definitions related to fractional calculus that is going to be used throughout the
paper.
\begin{definition}
  For $0<\alpha \leq 1$, and $t \geq 0$, the Riemann-Liouville integral of order $\alpha$ is
$$
{ }_0 \mathcal{D}_t^{-\alpha} y(t)=\frac{1}{\Gamma(\alpha)} \int_0^t \frac{y(\tau)}{(t-\tau)^{1-\alpha}} d \tau .
$$
where the Euler gamma function $\Gamma(\cdot)$ is defined by
$$
\Gamma(z)=\int_0^{\infty} t^{z-1} e^{-t} d t \quad(\mathbb{R}(z)>0) .
$$
\end{definition}
\begin{definition} For $0<\alpha \leq 1$, and $t \geq 0$, the Riemann-Liouville derivative of order $1-\alpha$ is
$$
{ }_0 \mathcal{D}_t^{1-\alpha} y(t)=\frac{1}{\Gamma(\alpha)} \frac{d}{d t} \int_0^t \frac{y(\tau)}{(t-\tau)^{1-\alpha}} d \tau .
$$
\end{definition}
\begin{definition}
 For $t \geq 0$, the Caputo derivative of order $\alpha$ is
$$
 { }^{C} \mathcal{D}_t^{\alpha} f(t)= \begin{cases}f^{(n)}(t) & \text { if } \alpha=n \in \mathbb{N}, \\ \frac{1}{\Gamma(n-\alpha)} \int_0^t \frac{f^{(n)}(x)}{(t-x)^{\alpha-n+1}} d x & \text { if } n-1<\alpha<n,\end{cases}
$$

\end{definition}

\begin{theorem}
     The Caputo fractional derivative of the power function satisfies
$$
^{C}D^\alpha t^p= \begin{cases}\frac{\Gamma(p+1)}{\Gamma(p-\alpha+1)} t^{p-\alpha}  & n-1<\alpha<n, p>n-1, p \in \mathbb{R}, \\ 0, & n-1<\alpha<n, p \leq n-1, p \in \mathbb{N} .\end{cases}
$$
\end{theorem}

\begin{definition}(\cite{Khalil}) The conformable fractional derivative of order
$\alpha\in (0,1)$
of a function $x(.)$ for $t>0$ is defined by
$$
T_{\alpha}x(t)=\displaystyle{\lim_{\varepsilon\longrightarrow0}\frac{x(t+\varepsilon t^{1-\alpha})-x(t)}{\varepsilon}},\hspace{0.2cm}t>0,$$
$$
T_{\alpha}x(0)=\displaystyle{\lim_{t\longrightarrow0^{+}}T_{\alpha}x(t)},
$$
provided that the limits exist.
\end{definition}
The fractional integral $I^{\alpha}(.)$ associated with the conformable fractional
derivative is defined by
 $$I^{\alpha}(x)(t)=\int_{0}^{t}s^{\alpha-1}x(s)ds.$$
\begin{theorem}(\cite{Khalil}) For a continuous function $x(.)$ in the domain of
$I^{\alpha}(.)$, we have
 $$T_{\alpha}(I^{\alpha}(x)(t))=x(t).$$
\end{theorem}

\begin{definition}
 If $\alpha \in \mathbb{R}, \notin \mathbb{Z}$, then a generalized fractional power series for a real function $y(t)$ is an infinite series of the form
 $$
  y(t)=\sum_{i \geq 0, j \geq 0}^{\infty} C_{i, j} t^{i \alpha+j}
$$
\end{definition}

\begin{definition} 
If $\alpha\in \mathbb{R}^{+}$, and $z \in \mathbb{R}$, then
$$
E_{\alpha}(z)=\sum_{n=0}^{\infty} \frac{z^n}{\Gamma(\alpha n+1)}
$$
defines the  Mittag-Leffler function.
\end{definition}
\begin{definition} 
If $\alpha, \beta \in \mathbb{R}^{+}$, and $z \in \mathbb{R}$, then
$$
E_{\alpha, \beta}(z)=\sum_{n=0}^{\infty} \frac{z^n}{\Gamma(\alpha n+\beta)}
$$
defines the generalized Mittag-Leffler function.
\end{definition}
\begin{definition}
If $\alpha, \beta$ and $ \gamma$ $ \in \mathbb{R}^{+}$, and $z \in \mathbb{C}$\\
$$
E_{\alpha, \beta, \gamma}(z)=1+\sum_{k=1}^{\infty} c_k z_k=1+\sum_{k=1}^{\infty}\left[\prod_{i=0}^{k-1} \frac{\Gamma[\alpha(i \beta+\gamma)+1]}{\Gamma[\alpha(i \beta+\gamma+1)+1]}\right] z^k
$$

 Defined    the Mittag-Leffler of tree index, this entire function, coincides for $\beta=1$ with the  generalized Mittag-Leffler function 
$$
E_{\alpha, 1, \gamma}(z)=\Gamma(\alpha \gamma+1) E_{\alpha, \alpha \gamma+1}(z),  $$ 
\end{definition}
\begin{definition}
     If $a_i, b_j, \alpha_i, \beta_j, \in \mathbb{R}^{+}$, and $z \in \mathbb{R}$ then
$$
{ }_p \Psi_q\left[\begin{array}{c}
\left(a_i, \alpha_i\right)_{1, p} \\
\left(b_j, \beta_j\right)_{1, q}
\end{array}\Bigr|z\right]=\sum_{k=0}^{\infty} \frac{\prod_{i=1, p} \Gamma\left(a_i+\alpha_i k\right)}{\prod_{j=1, q} \Gamma\left(b_j+\beta_j k\right)} \frac{z^k}{k !}
$$
Then, the standard Wright function  being obtained from (2.7) when $p=0$ and $q=1$ with $\beta_{1}=\lambda>-1, b_1=\mu$, reads
$$
W_{\lambda, \mu}(z) \equiv{ }_0 \Psi_1\left[\begin{array}{c}
-- \\
(\mu, \lambda)
\end{array} \Bigr| z\right] .
$$\\

Defined the generalized Wright function.\\
The generalized Mittag-Leffler function is a special case, $E_{\alpha, \beta}(z)={ }_1 \Psi_1\left[\begin{array}{c}(1,1) \\ (\beta, \alpha)\end{array}\Bigr|z\right] $. \\

\end{definition}

\section{Description of the method of a generalized fractional power series solutions for fractional differential equations.}

The Generalized Fractional Power Series Solution for Fractional Differential Equations can be expressed as an infinite series in the form :
\begin{align}
y(t) = \sum_{i \geq 0, j \geq 0}^{\infty} c_{i, j} t^{i \alpha + j}, \quad \alpha \in (0, 1).
\end{align}

A Cauchy product generalized fractional power series for a real function $y(t)$ is an infinite series of the form:
\begin{align}
y(t) = \sum_{i \geq 0, j \geq 0}^{\infty} c_{i, j} t^{i \alpha + j} = \left( \sum_{i=0}^{\infty} a_i t^{i \alpha} \right) \left( \sum_{j=0}^{\infty} b_j t^j \right),
\end{align}
where $c_{i, j} = a_i b_j$ for all $i$ and $j$.\\
Generalized Cauchy product fractional power series were introduced recently [19] to solve linear inhomogeneous differential equations with fractional-order Caputo derivatives and constant coefficients. These series can also be applied more broadly to solve linear fractional order differential equations, including those with variable coefficients, when the series are not constructed as Cauchy products.\\
The generalized fractional power series serves as a fundamental approach to solving linear fractional order differential equations. For a rational order $\alpha = \frac{p}{q}$ where $p \leq q \in \mathbb{N}$, the series (3.1) can be utilized by summing over all $i$ except $j = kq$ for $k \in \mathbb{N}$. This enables the series to be transformed into a regular power series in $z = t^{\frac{1}{q}}$. In practical applications, as demonstrated below, it's unnecessary to restrict the sum to the minimal set of independent functions in the trial ansatz.
 \\
  
 This simple example is sufficient to illustrate the wider applicability of the method:
\begin{example}{(\cite{Angstmann1})}
    Let $ 0 < \alpha< 1$ and $ b_{0}\in\mathbb{R}$. then the fractional differential equation
    
\begin{align}
_0\mathcal{D}_t^{1-\alpha} y(t) & = y^{'}(t) 
\end{align}
with the initial condition $ y(0) =b_{0} $ and $\lim\limits_{t \to 0} {}_0\mathcal{D}_t^{1-\alpha} y(t)= 0 $ has its solution given by

$$y(t)=b_{0}E_{\alpha}\left(\ t^{\alpha}\right).$$
\end{example}

We assume that a convergent series solution exists and we can apply the Riemann-Liouville derivative  term by term to write
\begin{align} 
\nonumber _0\mathcal{D}_t^{1-\alpha} \sum_{\substack{ i \geq 0, j \geq 0}}^{\infty} C_{i, j}  t^{i \alpha+j} =\sum_{\substack{i+j \geq 0\\ i \geq 0, j \geq 0}}^{\infty} C_{i, j} \frac{\Gamma(i\alpha+j+1)}{\Gamma((i+1) \alpha+j+1)} t^{(i+1) \alpha+j-1}\\
\end{align}
We re-label summation indices to arrive
at the balance equations,
\begin{align}
\sum_{\substack{i+j \geq 1 \\ i \geq 0, j \geq 0}}^{\infty} c_{i, j}(i \alpha+j) t^{i\alpha+j-1}-\sum_{i \geq 1, j \geq 0}^{\infty} c_{i-1, j} \frac{\Gamma((i-1) \alpha+j+1)}{\Gamma(i \alpha+j)} t^{i\alpha+j-1}=0 .
\end{align}
We can immediately deduce that
\begin{align}
c_{0, j}&=0 \quad \forall j \geq 1 .\\
c_{i, j}&=\frac{\Gamma((i-1) \alpha+j+1)}{\Gamma(i \alpha+j+1)} c_{i-1, j}  
 \quad   i\geq1 ,j\geq0
\end{align}

It follows from (3.6) and from recursive applications of (3.7) that  $c_{i, j}=0$ for all $i \geq 0$ and $j \geq 1$. Thus the only remaining non-zero coefficients are those with  $j=0$, and the recurrence relation then simplifies to
\begin{align}
c_{i, 0}=\frac{\Gamma((i-1) \alpha+1)}{\Gamma(i \alpha+1)} c_{i-1,0}
\end{align}
Equivalently we can write
\begin{align}
c_{i, 0}=c_{0,0} \frac{1}{\Gamma(i \alpha+1)} .
\end{align}
The series solution can now be written as
\begin{align}
y(t)=b_{0} \sum_{i=0}^{\infty} \frac{t^{i\alpha}}{\Gamma(i \alpha+1)}=b_{0}E_{\alpha}\left(t^\alpha\right).
\end{align}

The generalized fractional power series method is applicable to linear fractional order differential equations with constant and variable coefficients (see, for example, \cite{Angstmann1},\cite{Jaradat}). However, this method faces limitations, especially when there are more than two fractional derivatives with distinct orders in the equation, as well as when one of its terms is multiplied by $t^{\alpha}$ for example.This method is a particular case of our algorithm, which generalizes both fractional and classical power series.

\section{  The algorithm for solving fractional  and classical Differential Equations  }

Here, we explain the algorithm for solving fractional and classical differential equations. The algorithm is based on a generalized fractional power series defined as follows:
 
 \begin{align}
  y(t)=\sum_{i_{0}=0 }^{\infty}\sum_{i_{1}=0 }^{\infty}...\sum_{i_{n}=0 }^{\infty} C_{i_{0},... ,i_{n}} t^{\sum_{k=0}^{n}i_{k} \alpha_{k}},
  \quad  \alpha_{k} \in \mathbb{R};\notin \mathbb{N}.
\end{align}

where $n \in \mathbb{N}, t \geq 0$ is a variable of indeterminate, and $ C_{i_{0},... ,i_{n}}$ are the coefficients of the series.\\

Conveniently, we here assumed that the center of NGFPS (4.1) is zero since this can always be done via the linear change of variable $\left(t-t_0\right) \mapsto t$.\\

  To expand a fractional power series solution with $n$ fractional indices, the algorithm selects the number of indices in the exponent based on the distinct orders of fractional derivatives $D^{\alpha_{k}}$, where $\alpha_{k} \in \mathbb{R}$;$\notin \mathbb{N}$ for all $k \in \{1, \ldots, n\}$. Furthermore, it accommodates situations where terms in the equation are multiplied by $t^{\alpha_{k}}$. This means that for the fractional derivative of order $\alpha_k$, we include $i_k\alpha_k$ in the exponent, as well as when one of its terms is multiplied by $t^{\alpha_k}$. For the classical derivative, we include $i_{0}\alpha_{0}$ in the exponent with $\alpha_{0}=1$.

To complete the description of the basic algorithm and its most important properties, we should provide some examples to illustrate the algorithm:

\begin{example}
    Consider the relaxation  fractional differential equation
    \begin{align}
        ^{c}\mathcal{D}_t^{\alpha} y(t) = y(t),
    \end{align}
    where \( 0 < \alpha < 1 \) and \( y(0) = y_{0} \). The solution to this equation is given by
    \begin{align}
        y(t) = y_{0} E_{\alpha}\left( t^{\alpha} \right),
    \end{align}  
\end{example}

Let us consider the following  fractional power series of  y(t) ,is an infinite series of the form\\
\begin{align}
  y(t)=\sum_{i \geq 0}^{\infty} c_{i} t^{i \alpha}  
\end{align} 
The choice of $i\alpha$ into to power depends on the fractional derivative $D^{\alpha}$.We assume that a convergent series solution exists and we can apply the fractional derivative  term by term to write
\begin{align} 
\nonumber \mathcal{D}^{\alpha} \sum_{\substack{ i \geq 0}}^{\infty} c_{i}  t^{i \alpha} =\sum_{\substack{i \geq 1}}^{\infty} c_{i} \frac{\Gamma(i\alpha+1)}{\Gamma((i-1) \alpha+1)} t^{(i-1) \alpha}\\
\end{align}
 We re-label summation indices to arrive
at the balance equations,
\begin{align}
\sum_{\substack{i \geq 1}}^{\infty} c_{i} \frac{\Gamma(i\alpha+1)}{\Gamma((i-1) \alpha+1)} t^{(i-1) \alpha} -\sum_{i \geq 1}^{\infty} c_{i-1} t^{(i-1)}=0 .
\end{align}

 We can immediately deduce that
\begin{align}
c_{i}&=\frac{\Gamma((i-1) \alpha+1)}{\Gamma(i \alpha+1)} c_{i-1,}  
 \quad   i\geq1 
\end{align}

 Equivalently we can write
\begin{align}
c_{i}=c_{0} \frac{1}{\Gamma(i \alpha+1)} \quad
\quad with \quad c_{0}=y(0) .
\end{align}
 The series solution can now be written as
\begin{align}
y(t)=y_{0} \sum_{i=0}^{\infty} \frac{t^{i\alpha}}{\Gamma(i \alpha+1)}=y_{0}E_{\alpha}\left(t^\alpha\right).
\end{align}

Let's turn to another example. It is about the fractional oscillation equation, a topic extensively discussed by numerous authors(see \cite{Mainardi};\cite{Lin} ) using the Laplace transform to solve it, which is here solved with a new generalized fractional power series as follows:

\begin{example}
     the fractional oscillation equation
    \begin{align}
        ^{C}D^\alpha y(t) + w^{2}y(t) = 0, \quad 1 < \alpha \leq 2,
    \end{align}
    with initial conditions
    \begin{align}
        y(0) = c_0 \in \mathbb{R}, \quad y'(0) = c_1 \in \mathbb{R}.
    \end{align}
   has its solution given by
    \begin{align}
        y(t) = c_{0} E_{\alpha, 1}\left(-w^{2}t^\alpha\right) + c_{1} t E_{\alpha, 2}\left(-w^{2}t^\alpha\right),
    \end{align}
   
\end{example}

\begin{proof}
Let us consider the following generalized fractional power series of  y(t) ,is an infinite series of the form\\
\begin{align}
 y(t)=\sum_{i \geq 0, j \geq 0}^{\infty} c_{i, j} t^{i \alpha+j}  \end{align} 
 
The choice of \(i\alpha\) as the exponent depends on the fractional derivative \(D^{\alpha}\), and the integer $j$ for the initial value depends on  the classical derivative. Assuming a convergent series solution exists, we can apply the fractional derivative term by term to write

\begin{align} 
\nonumber ^{C}D^\alpha(\sum_{i \geq 0, j \geq 0}^{\infty} c_{i, j} t^{i \alpha+j} )
\nonumber &=\sum_{\substack{j \geq 2}}^{\infty} c_{0, j} \frac{\Gamma(j+1)}{\Gamma(-\alpha+j+1)} t^{- \alpha+j} +\sum_{\substack{ i \geq 1, j \geq 0}}^{\infty} c_{i, j} \frac{\Gamma(i\alpha+j+1)}{\Gamma((i-1) \alpha+j+1)} t^{(i-1) \alpha+j}\\
\end{align}
We have also 
\begin{align}
\nonumber w^{2}y(t)
\nonumber&=w^{2}\sum_{ i \geq 0, j \geq 0}^{\infty} c_{i, j} t^{i \alpha+j }\\
\nonumber &=w^{2}\sum_{ i \geq 1, j \geq 0}^{\infty} c_{i-1, j} t^{(i-1) \alpha+j}\\
\end{align}
We  re-label summation indices to arrive at the balance equations
\begin{align}
\nonumber &\sum_{\substack{ i \geq 1, j \geq 0}}^{\infty} C_{i, j} \frac{\Gamma(i\alpha+j+1)}{\Gamma((i-1) \alpha+j+1)} t^{(i-1) \alpha+j}
-w^{2}\sum_{ i \geq 1, j \geq 0}^{\infty} c_{i-1, j} t^{(i-1) \alpha+j}= 0\\
\end{align}
We can immediately deduce that
\begin{align}
c_{0, j} & =0 \quad j \geq 2, \\
c_{i, j} & =- w^{2}\frac{\Gamma((i-1)\alpha+j+1)}{ \Gamma(i \alpha+j+1)} c_{i-1, j} \quad i \geq 1, j \geq 0.
\end{align}
It further follows from (4.18) and from recursive applications of (4.19) that $c_{i, j}=0$ for all $i \geq 0, j\geq 2$. Thus the only remaining non-zero coefficients are those with $j<2$  and we then have the solution of the form
\begin{align}
y(t)=\sum_{j=0}^{1}\sum_{i=0}^{\infty} c_{i,j}t^{i\alpha+j}=\sum_{j=0}^{1}t^{j}\sum_{i=0}^{\infty} c_{i,j}t^{i\alpha}
\end{align}
  Then we  have 
\begin{align}
c_{i,j}=-w^{2}\frac{\Gamma((i-1)\alpha+j+1)}{\Gamma(i\alpha+j+1)} c_{i-1,j} \quad i \geq 1,  j \in \{0, 1\} .
\end{align}

Equivalently we can write
\begin{align}
c_{i,j}=(-1)^{i}\frac{w^{2i}\Gamma(j+1)}{\Gamma(i\alpha+j+1)} c_{0,j} \quad i \geq 1, \quad  j  \in \{0, 1\} .
\end{align}
where $c_{0,j}=\frac{y^{(j)}(0)}{j!}$ and we then have

\begin{align}
c_{i,j}=(-1)^{i}\frac{w^{2i}}{\Gamma(i\alpha+j+1)} y^{(j)}(0)\quad i \geq 1,  j \in \{0, 1\} .
\end{align}

Then the series solution can now be written as
\begin{align}
y(t) =
 c_{0}E_{\alpha, 1}\left(-w^{2}t^\alpha\right)+ c_{1}tE_{\alpha, 2}\left(-w^{2}t^\alpha\right) .
\end{align} 
\end{proof}

One of the novelties of this paper lies in utilizing a generalized fractional power series to solve a type of  classical Differential
Equations. To further illustrate the algorithm, we propose an equation previously solved by Saigo and Kibas(\cite{Saigo4}), leveraging the properties of the Mittag-Leffler function with three indices as follows :

\begin{theorem}
    Let \( n \in \mathbb{N} \) and $\beta\in \mathbb{R}^{+}, \notin \mathbb{N}$. Consider the initial value problem defined by the differential equation
    \begin{align}
        y^{(n)}(t) &= b t^\beta y(t), \quad 0 \leq t < \infty, \quad b \neq 0, \\
        y^{(k-1)}(0) &= c_k, \quad k = 1, 2, \ldots, n,
    \end{align}
    where \( c_k \in \mathbb{R} \). This problem admits an absolutely convergent generalized fractional power series solution given by
    \begin{align}
        y(t) = \sum_{j=1}^n \frac{c_j}{(j-1)!} t^{j-1} E_{n, 1+\frac{\beta}{n}, \frac{\beta+j-1}{n}}\left(b t^{\beta+n}\right),
    \end{align}
\end{theorem}

\begin{proof}
Let us consider the following generalized fractional power series of  y(t) :\\
\begin{align}
  y(t)=\sum_{i \geq 0, j \geq 0}^{\infty} c_{i, j} t^{i \beta+j}  
\end{align}

The choice of \(i\beta\) and \(j\) depends on the monomial \(t^{\beta}\) and the classical derivative in (3.2). This means that when we multiply the monomial \(t^{\beta}\) by the power series solution \(y(t)\), we observe a change in the index \(i\) as follows\\
\begin{align}
\nonumber t^{\beta}y(t)
\nonumber&=\sum_{ i \geq 0, j \geq 0}^{\infty} c_{i, j} t^{(i+1) \beta+j}\\
\nonumber &=\sum_{ i \geq 1, j \geq n}^{\infty} c_{i-1, j-n} t^{i \beta+j-n}\\
\end{align}
Furthermore, the classical derivative of the solution \(y(t)\) also causes a change at the level of the index \(j\). Specifically, when we take the \(n\)-th derivative of the power series solution \(y(t)\), the index \(j\) is shifted as follows:\\
 
\begin{align} 
\nonumber y^{(n)}(t)
\nonumber &=\sum_{i\ge 0,j\ge 0 }^{\infty} C_{i,j}(i\beta+j)(i\beta+j-1)......(i\beta-n+1)t^{i\beta+j-n}\\
\end{align}

\begin{align}
\nonumber y^{(n)}(t)
\nonumber &=\sum_{\substack{ i \geq 1, j \geq n}}^{\infty} C_{i, j} \frac{\Gamma(i \beta+j+1)}{\Gamma(i \beta+j-n+1)} t^{i \beta+j-n}
+\sum_{\substack{  j \geq n}}^{\infty} C_{0, j} \frac{\Gamma(j+1)}{\Gamma(j-n+1)} t^{j-n}\\
\nonumber &+\sum_{\substack{j=0}}^{n-1}\sum_{\substack{ i \geq 1}}^{\infty} C_{i, j} \frac{\Gamma(i \beta+j+1)}{\Gamma(i \beta+j-n+1)} t^{i \beta+j-n}\\
\end{align}

We then substitute (3.6) into (3.8)
and re-label summation indices to arrive at the balance equations
\begin{align}
\nonumber &\sum_{\substack{ i \geq 0, j \geq 0}}^{\infty} C_{i, j} \frac{\Gamma(i \beta+j+1)}{\Gamma(i \beta+j-n+1)} t^{i \beta+j-n} 
-b \sum_{ i \geq 1, j \geq n}^{\infty} c_{i-1, j-n} t^{i \beta+j-n} = 0\\
\end{align}
We can immediately deduce that
\begin{align}
c_{0, j} & =0 \quad j \geq n, \\
c_{i, j} & =0 \quad i \geq 1, \quad for \quad all\quad  j \in \{0,.....n-1\} \\
c_{i, j} & = b\frac{\Gamma(i\beta+j-n+1)}{ \Gamma(i \beta+j+1)} c_{i-1, j-n} \quad i \geq 1, j \geq n .
\end{align}
It further follows from (4.33) and from recursive applications of (4.35) that $c_{i, j}=0$ for all $i \geq 1, j>i n +n-1$. Similarly it follows from (4.34) with (4.35) that $c_{i, j}=0$ for all $j \geq 1, j<i n $. Thus the only remaining non-zero coefficients are those with  $ in\leq j\leq in+n-1 $   and we then have the solution of the form
\begin{align}
y(t)=\sum_{k=0}^{n-1}\sum_{i=0}^{\infty} c(i) t^{i\beta+in+k}=\sum_{k=0}^{n-1}t^{k}\sum_{i=0}^{\infty} c(i) t^{i(\beta+n)}
\end{align}
where $C(0)=c_{0,k}$ and the remaining $c(i,k)=c_{i, in+K}$ are found from the recurrence relation, (4.34),then\\
\begin{align}
C(i,k)=b\frac{\Gamma(i \beta+in+k-n+1)}{\Gamma(i \beta+in+k+1)} c(i-1,k) \quad i \geq 1,  \quad k \in \{0,.....n-1\} .
\end{align}

This can be solved in closed form yielding
\begin{align}
C(i,k)= \prod_{m=0}^{i-1}b\frac{\Gamma((m+1) (\beta+n)+k-n+1)}{\Gamma((m+1)( \beta+n)+k+1)} C(0,k) \quad i\geq 1, \quad  k \in \{0,.....n-1\}.
\end{align}

Equivalently we can write
\begin{align}
C(i,k)= \prod_{m=0}^{i-1}b\frac{\Gamma(n(m(\frac{\beta}{n}+1)+\frac{\beta}{n}+\frac{k}{n})+1)}{\Gamma(n(m(\frac{\beta}{n}+1)+\frac{\beta}{n}+1+\frac{k}{n})+1)}C(0,k) \quad i\geq 1,\quad k \in \{0,.....n-1\} .
\end{align}

Where $C(0,k)=c_{0,k}=\frac{y^{(k)}(0)}{k!}$ \\

The series solution can now be written as
\begin{align}y(t)=\sum_{k=0}^{n-1} \frac{y^{(k)}(0)}{k!} t^{k} E_{n, 1+\frac{\beta}{n},\frac{\beta+k}{n}}\left(b t^{\beta+n}\right)=\sum_{j=1}^n \frac{c_j}{(j-1)!} t^{j-1} E_{n, 1+\frac{\beta}{n}, \frac{\beta+j-1}{n}}\left(b t^{\beta+n}\right).
 \end{align}
\end{proof}

 \begin{remark}
      If $n \in \mathbb{N}, m>0$ and $l \in \mathbb{R}$  and let $\lambda \in \mathbb{C}$.\\
      Then the following differentiation formula
\begin{align}
\left(\frac{\mathrm{d}}{\mathrm{d} z}\right)^n\left[z^{n(l-m+1)} E_{n, m, l}\left(\lambda z^{n m}\right)\right]=\prod_{j=1}^n[n(l-m)+j] z^{n(l-m)}+\lambda z^{n l} E_{n, m, l}\left(\lambda z^{n m}\right)
\end{align}
holds. In particular, if
$$
n(l-m)=-j \text { for some } j=1,2, \cdots, n,
$$
then
\begin{align}
\left(\frac{\mathrm{d}}{\mathrm{d} z}\right)^n\left[z^{n(l-m+1)} E_{n, m, l}\left(\lambda z^{n m}\right)\right]=z^{n l} E_{n, m, l}\left(\lambda z^{n m}\right) \end{align}
 \end{remark}

The following  equation is used by the authors for modeling fractional models of anomalous relaxation based on the Kilbas and Saigo function(\cite{oliveira5}), which is traditionally solved using an ansatz series method, but in this study, we apply our algorithm.
 
\begin{proposition}
 Let  $\beta\in \mathbb{R}^{+}, \notin \mathbb{N}$,The fractional differential equation, initial value problem,defined by
\begin{align}
^{C}D^\alpha y(t) & =\lambda t^\beta y(t) \quad \quad   0 \leq t <\infty, \quad \lambda \neq 0; \\
y(0) & =y_{0}   ;
\end{align}
where $ 0 < \alpha < 1$,  has a generalized fractional power series
solution
\begin{align}
y(t)= y_{0}E_{\alpha, 1+\frac{\beta}{\alpha},\frac{\beta}{\alpha}}\left(\lambda t^{\alpha+\beta}\right)
\end{align}
\end{proposition}
\begin{proof}
Let us consider the following generalized fractional power series \\
\begin{align}  
y(t)=\sum_{i \geq 0, j \geq 0}^{\infty} c_{i, j} t^{i \alpha+j\beta}  
\end{align}
The choice of using $i\alpha+j\beta$ raised to the power depends on the monomial $t^{\beta}$ and the fractional derivative $D^{\alpha}$,we assume that a convergent series solution exists and we can apply the fractional derivative  term by term to write
\begin{align} 
\nonumber^{C}D^\alpha y(t) 
\nonumber &=\sum_{\substack{i+j \geq 1\\ i \geq 0, j \geq 0}}^{\infty} C_{i, j} \frac{\Gamma(i\alpha+j\beta+1)}{\Gamma((i-1) \alpha+j\beta+1)} t^{(i-1) \alpha+j\beta}\\
\end{align}
We have also 
\begin{align}
 t^{\beta}y(t)
=\sum_{ i \geq 0, j \geq 0}^{\infty} c_{i, j} t^{i \alpha+(j+1)\beta }
=\sum_{ i \geq 1, j \geq 1}^{\infty} c_{i-1, j-1} t^{(i-1) \alpha+j\beta}
\end{align}
We then substitute (4.46) and (4.47) into (4.42)
and re-label summation indices to arrive at the balance equations
\begin{align}
\nonumber &\sum_{\substack{i+j \geq 1\\ i \geq 0, j \geq 0}}^{\infty} C_{i, j} \frac{\Gamma(i\alpha+j\beta+1)}{\Gamma((i-1) \alpha+j\beta+1)} t^{(i-1) \alpha+j\beta}
-\lambda\sum_{ i \geq 1, j \geq 1}^{\infty} c_{i-1, j-1} t^{(i-1) \alpha+j\beta}= 0\\
\end{align}
Immediate deduction yields
\begin{align}
c_{0, j} & =0 \quad j \geq 1, \\
c_{i, 0} & =0 \quad i \geq 1,  \\
c_{i, j} & = \lambda\frac{\Gamma((i-1)\alpha+j\beta+1)}{ \Gamma(i \alpha+j\beta+1)} c_{i-1, j-1} \quad i \geq 1, j \geq 1.
\end{align}

From (4.49) and recursively from (4.51), $c_{i, j}=0$ for $i \geq 1, j>i$. Similarly, from (4.50) and (4.51), $c_{i, j}=0$ for $j \geq 1, j<i$. Thus, the only non-zero coefficients are $j=i$, yielding the solution of the forme :

\begin{align}
y(t)=\sum_{i=0}^{\infty} c(i) t^{i\alpha+i\beta}=\sum_{i=0}^{\infty} c(i) t^{i(\alpha+\beta)}
\end{align}
where $C(0)=c_{0,0}$ and the remaining $c(i)=c_{i, i}$ are found from the recurrence relation, (3.19) with $j=i$ and we then have the solution of the form
\begin{align}
C(i)=\lambda\frac{\Gamma((i-1)\alpha+j\beta+1)}{\Gamma(i\alpha+j\beta+1)} c(i-1) \quad i \geq 1, n \geq 1 .
\end{align}

This can be writen as 
\begin{align}
C(i)= \prod_{m=0}^{i-1}\lambda\frac{\Gamma(m (\alpha+\beta)+\beta+1)}{\Gamma(m(\alpha+ \beta)+\alpha+\beta+1)} C(0) \quad i\geq 1, j\geq 1 .
\end{align}

Equivalently we deduce that 
\begin{align}
C(i)= \prod_{m=0}^{i-1}\lambda\frac{\Gamma(\alpha(m(\frac{\beta}{\alpha}+1)+\frac{\beta}{\alpha})+1)}{\Gamma(\alpha(m(\frac{\beta}{\alpha}+1)+\frac{\beta}{\alpha}+1))}C(0) \quad i\geq 1, j \geq 1 .
\end{align}

Where $C(0)=c_{0,0}={y(0)}$ the series solution can now be written as

\begin{align}
    y(t)= y_{0}E_{\alpha, 1+\frac{\beta}{\alpha},\frac{\beta}{\alpha}}\left(\lambda t^{\alpha+\beta}\right).
    \end{align}
\end{proof}


To make our algorithm useful for more situations, we state the following theorem:\\

\begin{theorem}
 Let $ 0 < \beta < 1$ and $\nu\in \mathbb{R}^{+}, \notin \mathbb{N}$, then the fractional differential equation
\begin{align}
^{C}D^\beta\left(t^v \frac{d y(t)}{d t}\right)= t^{v-1} y(t)
\end{align}
with the initial condition $ y(0) =y_{0}$ has its solution given by
\begin{align}
y(t)=\Gamma(\nu)y_{0}W_{\beta, v}\left(\frac{t^\beta}{\beta}\right)
\end{align}
 \end{theorem}
 \begin{proof}
 We  suppose that the solution of this equation is in the form
 \begin{align}   y(t)=\sum_{\substack{i, j, k \geqslant 0 }}^{\infty} C_{i, j,k} t^{i \beta+j\nu+k}  
 \end{align}
 
 The choice of raising $i \beta+j\nu+k$ to the power depends on the fractional derivative $D^{\beta}$ and the monomial $t^{\nu}$, while the integer $k$ depends on the classical derivative.
 
One can easily obesrve that  
\begin{align}
t^{\nu-1}  y(t)=\sum_{\substack{i, j, k \geqslant 0 
}}^{+\infty}C_{i, j,k} t^{i\beta+(j+1)\nu+k-1} 
\end{align}
In the other hand, we know that
\begin{align}
t^\nu \frac{d y(t)}{d t}=\sum_{\substack{i, j, k \geqslant 0 \\
i+j+k \geqslant 1}}^{+\infty}C_{i, j,k}\left(i \beta+j\nu+k\right)  t^{i\beta+(j+1)\nu+k-1} 
\end{align}

    we can also observe that $t^\nu \frac{d y(t)}{d t}$ The conformable fractional derivative of order $1-\nu$ \\

Then one has
\begin{align}
 \frac{d^\beta}{d t^\beta}\left(t^\nu \frac{d y(t)}{d t}\right)=\sum_{\substack{i, j, k \geqslant 0 \\
i+j+k \geqslant 1}}^{+\infty}C_{i, j,k}\left(i \beta+j\nu+k\right) \frac{\Gamma(i \beta+(j+1) v+k)}{\Gamma((i-1) \beta+(j+1) v+k)} t^{(i-1)\beta+(j+1)\nu+k-1} 
\end{align}
We re-label summation indices to arrive at the balance equations:
\begin{align}
\sum_{\substack{i, j, k \geqslant 0 \\
i+j+k \geqslant 1}}^{+\infty}C_{i, j,k}\left(i \beta+j\nu+k\right) \frac{\Gamma(i \beta+(j+1) v+k)}{\Gamma((i-1) \beta+(j+1) v+k)} t^{(i-1)\beta+(j+1)\nu+k-1}-\sum_{\substack{i\geq1 j, k \geqslant 0 
}}^{+\infty}C_{i-1, j,k} t^{(i-1)\beta+(j+1)\nu+k-1} =0
\end{align}
We can immediately deduce that
\begin{align}
C_{0, j,k} & =0 \quad j+k \geq 1, \\
C_{i, j,k} & = \frac{\Gamma((i-1)\beta+(j+1)\nu+k)}{ (i\beta+j\nu+k)\Gamma(i \beta+(j+1)\nu+k)} c_{i-1, j,k} \quad i \geq 1, j \geq 0,k\geq0.
\end{align}
from (4.64) and (4.65) ; $c_{i,j,k }= 0 $ for all $i \geq 0 $ and $j+k \geq 1$. The only surviving term in the
series with powers in $j +k$ is $j+k = 0$, and the recurrence relation then simplifies to

\begin{align}
c_{i, 0,0}=\frac{1}{i\beta}\frac{\Gamma((i-1) \beta+\nu)}{\Gamma(i \beta+\nu)} c_{i-1,0,0}
\end{align}
where $C(0)=c_{0,0,0}=y(0)$ and the remaining $c(i)=c_{i, 0,0}$   and we then have 
\begin{align}
c(i) & = \frac{1}{i!\beta^{i}}\frac{\Gamma(\nu)}{\Gamma(i \beta+\nu)} c(0)
\end{align}
The series solution can now be written as
\begin{align}
y(t)=  y_{0}\sum_{i=0}^{\infty}\frac{1}{i!\beta^{i}}\frac{\Gamma(\nu)}{\Gamma(i \beta+\nu)}t^{i\beta}=\Gamma(\nu)y_{0}W_{\beta, v}\left(\frac{t^\beta}{\beta}\right)
\end{align}
\end{proof}

\begin{remark}
    By direct calculations we have that
$$
\begin{aligned}
\frac{d^\beta}{d t^\beta}\left(t^\nu \frac{d}{d t}\right) W_{\beta, v}\left(\frac{t^\beta}{\beta}\right) & =\frac{d^\beta}{d t^\beta} \sum_{k=1}^{\infty} \frac{t^{k \beta+v-1}}{(k-1)!\beta^{k-1}\Gamma(k \beta+v)} \\
& = \sum_{k=1}^{\infty} \frac{t^{k \beta+v-1-\beta}}{(k-1)!\beta^{k-1}\Gamma(k \beta+v-\beta)}= t^{\nu-1} W_{\beta, v}\left(\frac{t^\beta}{\beta^{}}\right).
\end{aligned}
$$
\end{remark}
 Let's turn to another example, which concerns the challenge of applying our algorithm as follows:
\begin{example}
    Let $ 0 < \alpha < 1$  and $\beta\in \mathbb{R}^{+}, \notin \mathbb{N}$. Consider the initial value problem defined by the differential equation
    \begin{align}
        t^{\beta}y^{"}(t)+T_{\alpha}y(t)+y(t) &= 0, \quad 0 \leq t < \infty, \\
        y(0) &= c_0, \quad  c_0 \in \mathbb{R},
        \\
        y^{'}(0) &= c_1, \quad  c_1 \in \mathbb{R},
    \end{align}
\end{example}
By applying our algorithm, we suppose that the solutions to this equation can be expressed as follows:
\begin{align}   y(t)=\sum_{\substack{i, j, k \geqslant 0 }}^{\infty} C_{i, j,k} t^{i\alpha+j\beta+k}  
 \end{align}
 
We then substitute (4.72) into (4.69) and  and equate cofficients of equal powers of t we obtain:
\newpage
\begin{align}
c_{i, 0,k-2} & = 0 & \text{for } i=1,2, \quad k \geq 2,\\
c_{i, j,k} & = 0 & \quad\text{for } k=0,1, \quad i \geq 1, j \geq 1,
\end{align}
\begin{equation}
\big((i-1)\alpha + (j-1)\beta + k\big)\big((i-1)\alpha + (j-1)\beta + k-1)\big) c_{i, j-1, k} + \big(i\alpha + j\beta + k-2\big) c_{i, j, k-2} + c_{i-1, j, k-2} = 0 \quad \text{for } i \geq 1, j \geq 1, k \geq 2
\end{equation}

The challenges in applying our algorithm to solve a special type of recurrence relation (4.75) with more than two indices are significant because there is no general procedure for solving such recurrence relations, which is why it is considered an art \cite{Knopfmacher}.That means when an equation contains more than two terms, it becomes more challenging to apply our algorithm, except in some particular cases. For example, when we consider example (4.3) with $ \beta =2-\alpha $,the initial value problem defined by:\\
\begin{theorem}
   Let $ 0 < \alpha < 1$  . Consider the initial value problem defined by the differential equation:
   \begin{align}
        t^{2-\alpha}y^{"}(t)+T_{\alpha}y(t)+y(t) &= 0, \quad 0 \leq t < \infty, \\
        y(0) &= c_0, \quad  c_0 \in \mathbb{R},
        \\
        y^{'}(0) &= c_1, \quad  c_1 \in \mathbb{R},
    \end{align}
 has its solution given by : 
\begin{align}
    y(t) =c_{0}\sum_{i=0}^{\infty}\frac{(-1)^{i}}{i!^{2}\alpha^{2i}}t^{i\alpha} +c_{1}\sum_{i=0}^{\infty}\frac{(-1)^{i}}{\prod_{m=1}^{i}(m\alpha+1)^{2}}t^{i\alpha+1}
\end{align}
\end{theorem}
\begin{proof}
Let us consider the following generalized fractional power series of  y(t) ,is an infinite series of the form\\
\begin{align}
 y(t)=\sum_{i \geq 0, j \geq 0}^{\infty} c_{i, j} t^{i \alpha+j}  
 \end{align} 
We then substitute (4.80) into (4.76) and  and equate cofficients of equal powers of t we obtain:
\begin{align}
c_{0, j} & =0 \quad j \geq 2, \\
c_{i, j} & = \frac{-1}{ (i \alpha+j)^{2}} c_{i-1, j} \quad i \geq 1, j \geq 1.
\end{align}

It further follows from (4.33) and from recursive applications of (4.35) that $c_{i, j}=0$ for all $i \geq 1, j>1$. Thus the only remaining non-zero coefficients are those with  $  j\leq 1 $   and we then have the solution of the form
\begin{align}
y(t)=\sum_{j=0}^{1}\sum_{i=0}^{\infty} c_{i,j} t^{i\alpha+j}=\sum_{j=0}^{1}t^{j}\sum_{i=0}^{\infty} c_{i,j}t^{i\alpha }
\end{align}
The recurrence relation(4.83) then simplifies to\\
\begin{align}
c_{i, j} = \frac{(-1)^{i}}{ \prod_{m=1}^{i}(m\alpha+j)^{2}} c_{0, j} \quad i \geq 1, j \geq 1.
\end{align}
Where $c_{0,j}=\frac{y^{(j)}(0)}{j!}$ 
The series solution is thus given by

\begin{align}
    y(t) =c_{0}\sum_{i=0}^{\infty}\frac{(-1)^{i}}{i!^{2}\alpha^{2i}}t^{i\alpha} +c_{1}\sum_{i=0}^{\infty}\frac{(-1)^{i}}{\prod_{m=1}^{i}(m\alpha+1)^{2}}t^{i\alpha+1}
\end{align}
\end{proof}

\section*{Conclusion and comments}

In this paper, we have explored the application of generalized fractional power series methods to solve linear fractional order differential equations with both constant and variable coefficients. We have identified the limitations of traditional methods such as the Laplace transform, particularly when dealing with inhomogeneous equations and variable coefficients. To address these challenges, we developed a new algorithm that extends the power series method to fractional differential equations, offering a more comprehensive and effective approach.\\

Our findings demonstrate that this new algorithm not only generalizes the classical and fractional power series method but also provides a robust solution framework for equations involving fractional derivatives. Through various examples, we illustrated the efficacy and versatility of our approach in solving complex differential equations. However, applying our algorithm to solve a special type of recurrence relation with more than two indices presents significant challenges because there is no general procedure for solving such recurrence relations. This implies that when an equation contains more than two terms, it becomes more challenging to apply our algorithm, except in some particular cases. \\

Overall, the contributions of this paper lie in extending the power series method to a broader class of differential equations, thereby opening new avenues for research and application in various scientific domains.

\section*{Data Availability}
No data were used to support this study.
\section*{Conflicts of Interest}
The author declares no conflicts of interest.
\section*{Acknowledgements}
This paper is dedicated to my late father, Assebbane Lahcen. The authors are deeply grateful to the anonymous referee for his several valuable and insightful suggestions which improved the quality of the present paper.




\begin{thebibliography}{99}
\bibitem{Angstmann1}\label{power series 1}Angstmann, Christopher N., and Bruce Ian Henry. "Generalized fractional power series solutions for fractional differential equations." Applied Mathematics Letters 102 (2020): 106107.
\bibitem{Ming-Fan}\label{series expansion}
Li, Ming-Fan, Ji-Rong Ren, and Tao Zhu. "Series expansion in fractional calculus and fractional differential equations." arXiv preprint arXiv:0910.4819 (2009).
\bibitem{Sirunya}\label{Integro-Differential}
Thanompolkrang, Sirunya, and Duangkamol Poltem. "A Generalized Fractional Power Series for Solving a Class of Nonlinear Fractional Integro-Differential Equation." (2018).
\bibitem{Saigo4}\label{Mittag-liffler }
Saigo, Megumi, and Anatolii Aleksandrovich Kilbas. "The solution of a class of linear differential equations via functions of the Mittag-Leffler type." Differential Equations 36 (2000): 193-202.
\bibitem{oliveira5}\label{relaxation}
de Oliveira, Edmundo Capelas, Francesco Mainardi, and Jayme Vaz Jr. "Fractional models of anomalous relaxation based on the Kilbas and Saigo function." Meccanica 49.9 (2014): 2049-2060.
\bibitem{Arikoglu}\label{Differential transform method}
Arikoglu, Aytac, and Ibrahim Ozkol. "Solution of difference equations by using differential transform method." Applied mathematics and computation 174.2 (2006): 1216-1228.
\bibitem{shawagfeh}\label{Decomposition method}
Shawagfeh, N. T. "The decomposition method for fractional differential equations." J. Frac. Calc 16 (1999): 27-33.
\bibitem{MOmani}\label{numerical comparison}
Momani, Shaher, and Zaid Odibat. "Numerical comparison of methods for solving linear differential equations of fractional order." Chaos, Solitons $\&$ Fractals 31.5 (2007): 1248-1255.
\bibitem{Lin}\label{Laplace transform}
Lin, Shy-Der, and Chia-Hung Lu. "Laplace transform for solving some families of fractional differential equations and its applications." Advances in Difference Equations 2013 (2013): 1-9.
\bibitem{Angstmann10}\label{biologie}
Angstmann, C. N., B. I. Henry, and A. V. McGann. "A fractional order recovery SIR model from a stochastic process." Bulletin of mathematical biology 78 (2016): 468-499.
\bibitem{Angstmann11}\label{Compartement models}
Angstmann, Christopher N., et al. "Fractional order compartment models." SIAM Journal on Applied Mathematics 77.2 (2017): 430-446.

\bibitem{Jaradat12}Jaradat, I., et al. "Theory and applications of a more general form for fractional power series expansion." Chaos, Solitons \& Fractals 108 (2018): 107-110.

\bibitem{Matlob13}\label{viscoelastic}
Matlob, Mohammad Amirian, and Yousef Jamali. "The concepts and applications of fractional order differential calculus in modeling of viscoelastic systems: A primer." Critical Reviews™ in Biomedical Engineering 47.4 (2019).
\bibitem{Henry14}\label{neuronal}
Henry, B. I., T. A. M. Langlands, and S. L. Wearne. "Fractional cable models for spiny neuronal dendrites." Physical review letters 100.12 (2008): 128103.
\bibitem{Tarasov15}\label{Media}
Tarasov, Vasily E. Fractional dynamics: applications of fractional calculus to dynamics of particles, fields and media. Springer Science $\&$ Business Media, 2011.
\bibitem{Tarasov}\label{econimics}
Tarasov, Vasily E. "On history of mathematical economics: Application of fractional calculus." Mathematics 7.6 (2019): 509.
\bibitem{Gosta}\label{Mittag-Liffler}
Mittag-Leffler, Gösta Magnus. "Sur la nouvelle fonction $E_{\alpha} (x)$." CR Acad. Sci. Paris 137.2 (1903): 554-558.
\bibitem{Salim}\label{generalizatio-mittag liffler}
Salim, Tariq O., and Ahmad W. Faraj. "A generalization of Mittag-Leffler function and integral operator associated with fractional calculus." J. Fract. Calc. Appl 3.5 (2012): 1-13.
\bibitem{Jaradat}\label{general form for fractional power}
Jaradat, I., et al. "Theory and applications of a more general form for fractional power series expansion." Chaos, Solitons $\&$ Fractals 108 (2018): 107-110.
\bibitem{Knopfmacher}\label{recurrence relation with two indices}Knopfmacher, Arnold, Toufik Mansour, and Augustine Munagi. "Recurrence relation with two indices and plane compositions." Journal of Difference Equations and Applications 17.1 (2011): 115-127.
\bibitem{Khalil}\label{A new definition of fractional derivative}R. Khalil, M. Al Horani, A. Yousef and M. Sababheh, A new definition of fractional derivative, Journal of Computational and Applied Mathematics.
\bibitem{Mainardi}\label{Fractional relaxation-oscillation and fractional }Mainardi, Francesco. "Fractional relaxation-oscillation and fractional diffusion-wave phenomena." Chaos, Solitons $\&$ Fractals 7.9 (1996): 1461-1477.
\end{thebibliography}
\end{document}